[1994/12/01]

\documentclass[a4paper, 12pt]{amsart}
\usepackage{amssymb,amsxtra}
\usepackage{amsmath,amsthm,amscd}

\chardef\bslash=`\\ 

\newtheorem{thm}{Theorem}[section]
\newtheorem{cor}[thm]{Corollary}
\newtheorem{lem}[thm]{Lemma}
\newtheorem{prop}[thm]{Proposition}

\theoremstyle{definition}
\newtheorem{defn}{Definition}[section]
\newtheorem{ex}{Example}[section]

\theoremstyle{remark}
\newtheorem{rem}{Remark}[section]
\newtheorem*{notation}{Notation}

\newcommand{\thmref}[1]{Theorem~\ref{#1}}

\newcommand{\lemref}[1]{Lemma~\ref{#1}}
\newcommand{\propref}[1]{Proposition~\ref{#1}}

\newcommand{\defnref}[1]{Definition~\ref{#1}}
\newcommand{\exref}[1]{Example~\ref{#1}}

\newcommand{\eval}[2][\right]{\relax
  \ifx#1\right\relax \left.\fi#2#1\rvert}

\def\projlim{\mathop{\oalign{lim\cr
\hidewidth$\longleftarrow$\hidewidth\cr}}}

 \def\R{{\bf{R}}}
\def\C{{\bf{C}}}
\def\Z{{\bf{Z}}}
\def\N{{\bf{N}}}

\def\t{{\bf{T}}}
\def\Q{{\bf{Q}}}

\def\tfae{the following conditions are equivalent:}

\title{Random walks on Bratteli diagrams}
\author{Jean Renault}
\address{Universit\'e d'Orl\'eans et CNRS  (UMR 7349 et FR2964), D\'epartement de Math\'ematiques, 
F-45067 Orl\'eans Cedex 2, France}
\email{jean.renault@univ-orleans.fr}

\usepackage{pdfsync}
\begin{document}
\vskip5mm
\begin{abstract}
In the eighties, A. Connes and E. J. Woods made a connection between  hyperfinite von Neumann algebras and Poisson boundaries of time dependent random walks. I will explain this connection and will present two theorems given there:  the construction of a large class of states on a hyperfinite von Neumann algebra (due to A. Connes) and the ergodic decomposition of a Markov measure via harmonic functions (a classical result in probability theory). The crux of the first theorem is a model for conditional expectations on finite dimensional C*-algebras. Our proof of the second theorem hinges on the notion of cotransition probability.
\end{abstract}
\maketitle
\markboth{Jean Renault}
{Random walks on Bratteli diagrams}

\renewcommand{\sectionmark}[1]{}

\section{Introduction.}

The connection between operator algebras and ergodic theory goes back to the early days of these subjects. More recently, A.~Connes and E.~J.~Woods have uncovered in \cite{cw:random} a new connection between hyperfinite von Neumann algebras and Markov chains. They have identified the flow of weights of an ITPFI factor $M$ as the Poisson boundary of a group-invariant time-dependent Markov random walk on the real line $\R$.  They also generalize this identification  to arbitrary hyperfinite factors by introducing matrix-valued random walks. In the same article, they give a partial answer to a related question in ergodic theory which goes back to G.~Mackey. Given a locally compact group $G$, characterize the measured $G$-spaces $X$ which can be constructed as the Poisson boundary of a random walk on $G$. A necessary condition is the amenability in the sense of Zimmer of the $G$-space $X$ (Zimmer had already shown the amenability of the Poisson boundary of a time-independent random walk). Another necessary condition is its approximate transitivity, a notion introduced earlier by Connes and Woods to characterize the flow of weights of an ITPFI factor. They state that these conditions are also sufficient and show that this is the case for transitive actions and when the group is $\R$ or $\Z$. More generally, one can ask what measured $G$-spaces  can be constructed as the Poisson boundary of a matrix-valued random walk on $G$. The complete answer is given by  Adams, Elliott and Giordano in \cite{aeg:continuous} (see also \cite{eg:discrete}): these are exactly the amenable $G$-spaces.

My purpose here is not to explain these results, but to put into light some points which are only implicit in \cite{cw:random}. I shall put on the front of the stage random walks on Bratteli diagrams, which is in fact another name for time-dependent Markov chains (with discrete time), and cotransition probabilities. While cotransition probabilities do not appear explicitly in the case of UHF diagrams, they become crucial when studying arbitrary Bratteli diagrams. Their importance in the study of  boundaries of  random walks on Bratteli diagrams is also emphasized in the recent work \cite{ver:equipped,ver:graded graphs} of A.~Vershik.  My scope will be limited to the presentation of two known theorems which play a key r\^ole in \cite{cw:random}. The first one, given here as \thmref{Markov states}, is Theorem 1 of \cite{ac:krieger}, which  makes a direct connection between a large class of states on a hyperfinite von Neumann algebra and random walks on a Bratteli diagram. The second one (\thmref{harmonic})  is a classical result in probability theory \cite[Proposition V-2-2]{nev:proba}  which gives the ergodic decomposition of a Markov measure via bounded harmonic functions. On the way, I will give the reduction of faithful conditional expectations on a finite dimensional C*-algebra (\thmref{fd expectation}), a definitely well-known result which can be easily extracted from \cite{ac:krieger}, because it provides the building block of a random walk on a Bratteli diagram. The definitions of a matrix-valued random walk on a group and of its Poisson boundary will be given in the last section only, where they should appear more natural after the presentation of \thmref{Markov states} and \thmref{harmonic}.

\section{Conditional expectations on finite dimensional C*-algebras}

Although the material of this section is not new, I have not found a reference for \thmref{fd expectation} below which gives a complete invariant for a faithful conditional expectation on a finite dimensional C*-algebra. The description of an inclusion of finite dimensional C*-algebras  in terms of matrix units, which is a part of the theorem, and its description by a diagram, have been given by O. Bratteli in his fundamental paper \cite{bra:AF} (see its proposition 1.7 and its section 1.8). The graphical description of a conditional expectation on a finite dimensional C*-algebra appears in section 3 (iii) of \cite{cw:random} but without much detail. Our proof will emphasize Cartan (or diagonal) subalgebras, which appear only implicitly in the work of Bratteli. The corresponding groupoid models are also known as path models or tail equivalence relations.

We first give the ingredients and the recipe to construct a faithful conditional expectation $Q$ of a finite dimensional C*-algebra $\underline M$ onto a sub-C*-algebra $M$. Then, we show that every faithful conditional expectation is obtained by this recipe. Let us recall that we can associate to an equivalence relation $R$ on a finite set $X$ a finite dimensional C*-algebra $M=C^*(R)$: its elements are the functions (or matrices) $f:R\to\C$ (here $R\subset X\times X$ is the graph of the equivalence relation), the product is the matrix multiplication and the involution is the usual complex conjugate of a matrix. It has a canonical matrix unit $(e(x,y))_{(x,y)\in R}$ indexed by $R$. Consider now finite sets $X,V,E,{\underline V}$ equipped with surjections $r:X\to V$, $s:E\to V$, $r:E\to{\underline V}$. The triple $(V,E,\underline V)$, which we call here a {\it graph}, will be ubiquitous in this survey. Define
$$\underline X=\{(x,a)\in X\times E: r(x)=s(a)\}$$
and the equivalence relations:
$$R=\{(x,y)\in X\times X: r(x)=r(y)\}$$
$$\underline R=\{(xa, yb)\in\underline X\times\underline X: r(a)=r(b).$$
We can construct the C*-algebras $C^*(R)$ and $C^*({\underline R})$. Moreover, the map $j: C^*(R)\to C^*({\underline R})$ given by
$$j(f)(xa,yb)=\left\{\begin{array}{ccc}
f(x,y) &{\rm if}& a=b\\
0 &{\rm if}& a\not=b
\end{array}\right.$$
identifies $C^*(R)$ to a subalgebra of $C^*(\underline R)$. We shall make this identification and view the elements of $C^*(R)$ as functions on $\underline R$. Then $(V,E,\underline V)$ is the graph of the inclusion. We leave as an exercise to the reader the proof of the following lemma:

\begin{lem} Let $R'$ be the following equivalence relation on $E$:
$$R'=\{(a,b)\in E\times E: s(a)=s(b), r(a)=r(b)\}$$
Then the map $k: C^*(R')\to C^*({\underline R})$ given by
$$k(g)(xa,yb)=\left\{\begin{array}{ccc}
g(a,b) &{\rm if}&x=y\\
0 &{\rm if}&x\not=y
\end{array}\right.$$
identifies $C^*(R')$ to the commutant of $C^*(R)$ in $C^*(\underline R)$.
\end{lem}

\begin{defn} A {\it transition probability} on the graph $E$ is a function $p:E\to\R_+^*$ such that for all $v\in V$, $\sum_{s(c)=v}p(c)=1$.
\end{defn}

\begin{prop}\label{model expectation} Let $X,V,E,\underline V$ be as above and let $p$ be a transition probability on $E$. The map $Q:C^*({\underline R})\to  C^*(R)$ defined by
$$Q({\underline f})(x,y)=\sum_c p(c) {\underline f}(xc, yc),$$
where the sum is over all edges $c\in E$ originating from the common range of $x$ and $y$, is 
a faithful conditional expectation onto $C^*(R)$.
\end{prop}
\begin{proof} This is a straighforward verification.
\end{proof}

We are going to prove a converse to \propref{model expectation}: namely all faithful conditional expectations $Q:\underline M\to M$, where $M$ is a sub-C*-algebra of a finite dimensional C*-algebra $\underline M$, are of that form. We first recall the notion of Cartan subalgebra which will be our main tool. It is an algebraic characterization of the canonical abelian subalgebra $C(X)$ of the C*-algebra $C^*(R)$ of an equivalence relation $R$ on $X$ as above.

\begin{defn} An abelian subalgebra $A$ of a von Neumann algebra $M$ is called a {\it Cartan subalgebra} if it is maximal self-adjoint, regular and there exists a faithful normal conditional expectation $P:M\to A$. We then say that $(M,A)$ is a {\it Cartan pair}.
\end{defn}
Regularity means that the normalizer of $A$ in $M$, which is defined here as
$$N_M(A)=\{ v\,\hbox{partial isometry of}\,M: vAv^*\subset A, v^*Av\subset A\},$$
generates $M$ as a von Neumann algebra. We recall the fact that the conditional expectation $P$ is unique. The main result of \cite{fm:relations} is that every Cartan pair $(M,A)$ (if one assumes that $M$ acts on a separable Hilbert space) is of the form $(W^*(R,\tau), L^\infty(X,\mu))$ where $R$ is a countable Borel equivalence relation on a standard measured space $(X,\mu)$, where $\mu$ is a quasi-invariant measure and $\tau\in Z^2(R,\t)$ is a Borel twist. When $M$ is finite dimensional, the result of \cite{fm:relations} is elementary: we let $X$ be the spectrum of $A$ and $V$ be the spectrum of the centre $Z(M)$ of $M$. The inclusion $Z(M)\subset A$ gives a surjective map $r:X\to V$. We let $R$ be the equivalence relation admitting $r$ as quotient map. Each $x\in X$ corresponds to a minimal projection $e(x)$ in $A$; $e(x)$ and $e(y)$ are equivalent if and only if $(x,y)\in R$. We choose a matrix unit $(e(x,y))_{(x,y)\in R}$ such that for all $x\in X$, $e(x,x)=e(x)$. This matrix unit defines an isomorphism $M\to C^*(R)$ sending $A$ to $C(X)$. Thus, when $M$ is finite dimensional, the twist is trivial. However, it does not admit a canonical trivialization. Note also that in a finite dimensional C*-algebra, the notions of Cartan subalgebra and of maximal abelian self-adjoint subalgebra agree.
We shall need an easy lemma about extension of matrix units.

\begin{lem}\label{partial} Let $(M,A)$ be a finite dimensional Cartan pair and let $(X,R)$ be the corresponding equivalence relation. Then every partial matrix unit $(e(x,y))_{(x,y)\in S}$ in $M$, where $S$ is a subequivalence relation of $R$, can be extended to a full matrix unit $(e(x,y))_{(x,y)\in R}$.
\end{lem}

\begin{proof} We fix an arbitrary full matrix unit $(\underline e(x,y), (x,y)\in R)$. There exists a function $c:S\to \t$, where $\t$ is the group of complex numbers of module 1, such that $e(x,y)=c(x,y)\underline e(x,y)$ for all $(x,y)\in S$. It is a cocycle. Every cocycle on $S$ is trivial: there exists $b:X\to\t$ such that $c(x,y)=b(x)\overline{b(y)}$ for all $(x,y)\in S$. Then, we define
$e(x,y)=b(x)\underline e(x,y)\overline{b(y)}$ for all $(x,y)\in R$.
\end{proof}

The following lemma is a complement to Lemma III.1.14 of \cite{ren:approach}.

\begin{lem}\label{key lemma} Given  an inclusion $M\subset {\underline M}$ of finite dimensional C*-algebras, a faithful conditional expectation $Q:{\underline M}\to M$ and a Cartan subalgebra $A$ of $M$, there exists a Cartan subalgebra ${\underline A}$ of ${\underline  M}$ such that 
\begin{enumerate} 
\item $A\subset {\underline A}$, 
\item $N_{M}( A)\subset N_{{\underline M}}({\underline A})$, and 
\item $Q\circ {\underline P}=P\circ Q_0$, where $P$ is the conditional expectation from $M$ onto $A$, ${\underline P}$ is the conditional expectation from ${\underline M}$ to ${\underline A}$ and $Q_0$ is the restriction of $Q$ to ${\underline A}$.
\end{enumerate}
\end{lem}

\begin{proof} We let $(X,R)$ be the equivalence relation defined by the pair $(M,A)$: $X$ is the spectrum of $A$, $V$ is the spectrum of the centre $Z(M)$ of $M$ and $r:X\to V$ is the quotient map. We choose a matrix unit  $(e(x,y))_{(x,y)}\in R$ of $M$ with $e(x,x)=e(x)$ minimal projection corresponding to $x$. We choose a section $\sigma$ for the map $r: X\to V$. For  each $v\in V$, we set ${\underline M}_v=e(\sigma(v)){\underline M}e(\sigma(v))$. There exists a unique state ${\underline \varphi}_v$ of the algebra ${\underline M}_v$ such that $Q({\underline f})={\underline \varphi}_v({\underline f})e(\sigma(v))$ for  all ${\underline f}\in {\underline M}_v$. It is faithful because $Q$ is faithful. Since self-adjoint matrices are diagonalizable, there exists a Cartan subalgebra ${\underline A}_v$ of ${\underline M}_v$ such that ${\underline \varphi}_v={\underline \varphi}_v\circ {\underline P}_v$, where ${\underline P}_v$ is the conditional expectation onto ${\underline A}_v$. For $x\in X$, we define
$${\underline A}_x=e(x,\sigma(r(x))){\underline A}_{r(x)}e(\sigma(r(x)),x)$$
Then ${\underline A}=\oplus_{x\in X} {\underline A}_x$ is a Cartan subalgebra of ${\underline M}$. It contains $A$ because for all $x\in X$, ${\underline A}_x$ contains $e(x)$ as its unit element. By construction $e(x,y)$ belongs to the normalizer of ${\underline A}$ in ${\underline M}$, hence $N_{M}( A)\subset N_{\underline M}({\underline A})$.
Let $v\in V$ and ${\underline a}\in {\underline M}_v$. Then
$$P\circ Q({\underline a})=P({\underline \varphi}_v({\underline a})e(\sigma(v)))={\underline \varphi}_v({\underline a})e(\sigma(v)).$$
On the other hand, since ${\underline P}_v$ is the restriction of  ${\underline P}$  to ${\underline M}_v$,
$$Q\circ{\underline P}({\underline a})={\underline \varphi}_v({\underline P}_v({\underline a}))e(\sigma(v))={\underline \varphi}_v({\underline a})e(\sigma(v)).$$
Thus, $P\circ Q$ and $Q\circ{\underline P}$ agree on ${\underline M}_v$ 
Suppose now that ${\underline a}$ belongs to $e(x){\underline M}e(y)$, where $(x,y)\in R$. We write ${\underline a}=e(x,v){\underline a}_v e(v,y)$ with ${\underline a}_v\in {\underline M}_v$. Since $e(x,v)$ and $e(v,x)$ belong to $M$, $Q({\underline a})=e(x,v)Q({\underline a}_v)e(v,y)$ and since $e(x,v)$ and $e(v,x)$ belong to $N_M(A)$,
$$P\circ Q({\underline a})=e(x,v)P\circ Q({\underline a}_v)e(v,y).$$
On the other hand, since $e(x,v)$ and $e(v,x)$ belong also to $N_{\underline M}({\underline A})$, ${\underline P}({\underline a})=e(x,v){\underline P}({\underline a}_v)e(v,y)$. Hence
$$Q\circ {\underline P}({\underline a})=e(x,v)Q\circ {\underline P}({\underline a}_v)e(v,y).$$
Therefore, $P\circ Q$ and $Q\circ{\underline P}$ agree on $e(x){\underline M}e(y)$. We deduce that they agree on $\underline M$. This implies that $Q({\underline A})=A$ and that we have the equality $Q\circ {\underline P}=P\circ Q_0$ where $Q_0$ is the restriction of $Q$ to $\underline A$.\\
\end{proof}

\begin{defn} Let $Q:\underline M\to M$ be a conditional expectation and let $A,\underline A$ be Cartan subalgebras of $M,\underline M$ respectively. We say that we have a {\it Cartan pairs inclusion}  if the conditions (i) and (ii) of the lemma are satisfied; we then write $(M,A)\subset (\underline M,\underline A)$. We say that the inclusion $(M,A)\subset (\underline M,\underline A)$ is {\it compatible} with $Q$ if the condition (iii) of the lemma is also satisfied.
\end{defn}

\begin{lem} Let $(M,A)\subset(\underline M,\underline A)$ be an inclusion of finite dimensional Cartan pairs. Then  the spectrum $\underline X$ of $\underline A$ is canonically identified to the fibered product $X\times_VE$, where $X$ is the spectrum of $A$, $E$ is the spectrum of $M'\cap {\underline A}$, $V$ is the spectrum of $Z(M)$ and the fibered product is relative to  the maps $r:X\to V$ and $s:E\to V$ given by the inclusions $Z(M)\subset M$ and $Z(M)\subset M'\cap {\underline A}$.
\end{lem}

\begin{proof} We let $\alpha$ be the action of $N_M(A)$ on $X$ and $\underline\alpha$ be the action of $N_{\underline M}(\underline A)$ on $\underline X$. We let $\pi:\underline X\to X$ and $\pi:\underline X\to E$ be the surjections corresponding to the inclusions $A\subset\underline A$ and $M'\cap\underline A\subset\underline A$. They satisfy $r\circ\pi=s\circ q$. Hence $(\pi,q)$ maps $\underline X$ into the fibered product  $X\times_VE$. This map is injective: let $\underline x,\underline y\in\underline X$ such that $\pi(\underline x)=\pi(\underline y)$ and $q(\underline x)=q(\underline y)$. The elements of $M'\cap\underline A$ are exactly the functions on $\underline X$ which are constant under the action $\underline\alpha$ of $N_M(A)$ on $\underline X$. Therefore, the relation $q(\underline x)=q(\underline y)$ implies the existence of $u\in N_M(A)$ such that $\underline y=\underline\alpha_u(\underline x)$. This implies that $\pi(\underline x)=\pi(\underline y)=\alpha_u(\pi(\underline x))$, hence $ue(x)=e(x)$ and $\underline y=\underline x$. The map is surjective. Let $(x,c)\in X\times E$ such that $r(x)=s(c)$. Pick $\underline y\in\underline X$ such that $q(\underline y)=c$. Since $r(\pi(\underline y))=r(x)$, there exists $u\in N_M(A)$ such that $x=\alpha_u(\pi(\underline y))$. Then $\underline x=\underline\alpha_u(\underline y)$ does the job.
\end{proof}

An equivalent statement of the lemma is that $\underline A$ is canonically identified to $A\otimes_{Z(M)}(M'\cap {\underline A})$.

\begin{lem}\label{commutant} Let $(M,A)\subset(\underline M,\underline A)$ be an inclusion of finite dimensional Cartan pairs. The commutant of $M$ in $\underline M$ is denoted by $M'$.
\begin{enumerate}
\item If $\underline a$ belongs to $M'$, then
$\underline e(xa)\underline a\underline e(yb)=0$ if  $x\not=y$.
\item $M'\cap\underline A$ is a Cartan subalgebra of $M'$.
\item the equivalence relation induced on the spectrum $E$ of $M'\cap\underline A$ by the normalizer is
$$R'=\{(a,b)\in E\times E: s(a)=s(b), r(a)=r(b)\}$$

\end{enumerate}
\end{lem}

\begin{proof} i) Assume that $\underline f$ commutes with $M$. If $x\not=y$,
$$ \underline e(xa)\underline f\underline e(yb)= \underline e(xa)e(x)\underline f\underline e(yb)= \underline e(xa)\underline f e(x)\underline e(yb)=0.$$
ii) For $c\in E$, we denote by $\epsilon(c)$ the corresponding projection in $M'\cap\underline A$. According to (i), $\epsilon(c)=\sum_{r(x)=s(c)} \underline e(xc)$. Suppose that $\underline f\in M'$ commutes with the elements of $M'\cap\underline A$. Consider $xa$ and $xb$ with $a\not=b$. Then
$$ \underline e(xa)\underline f\underline e(xb)= \underline e(xa)\epsilon(a)\underline f\underline e(yb)= \underline e(xa)\underline f \epsilon(a)\underline e(yb)=0.$$
Thus $\underline e(xa)\underline f\underline e(yb)=0$ if $xa\not=yb$, therefore $\underline f$ belongs to $\underline A$.\\
(iii) Assume that $\epsilon(a)M'\epsilon(b)\not=0$. Then according to (i), there exists $x\in X$ such that $(xa,xb)\in \underline R$. This implies that $(a,b)\in R'$. Conversely, if $(a,b)\in R'$, we pick $x\in X$ such that $r(x)=s(a)=s(b)$. We choose a partial isometry $\underline u\in\underline M$ such that $\underline u\underline u^*=\underline e(xa)$, $\underline u^*\underline u=\underline e(xb)$. Then $\sum e(y,x)\underline u e(x,y)$, where $(e(x,y))_{(x,y)\in R}$ is a matrix unit for $R$ and the sum is over the $y$'s such that $r(y)=r(x)$ is a partial isometry in $M'$ with domain $\epsilon(b)$ and range $\epsilon(a)$.

\end{proof}

\begin{lem}\label{transition probability} Let $(M,A)\subset(\underline M,\underline A)$ be an inclusion of finite dimensional Cartan pairs and let $Q:\underline M\to M$ be a faithful conditional expectation which satisfies the condition (iii) of \lemref{key lemma}.
 Then, with above notations, there exists  a transition probability $p$ on the graph $E$ such that for all $c\in E$,
$$Q(\epsilon(c))=p(c)e(s(c))$$
where $\epsilon(c)$ is the minimal projection in $M'\cap\underline A$ corresponding to $c\in E$ and $e(v)$ is the minimal projection in $Z(M)$ corresponding to $v\in V$. 
\end{lem}

\begin{proof} As earlier, we denote by $Q_0$ the restriction of $Q$ to $\underline A$. We first check that $Q_0(\epsilon(c))$ belongs to $Z(M)$: for $a\in M$, 
$$aQ_0(\epsilon(c))=Q_0(a\epsilon(c))=Q_0(\epsilon(c)a)=Q_0(\epsilon(c))a$$
Then, we observe that $e(x)\epsilon(c)=0$ if $r(x)\not=s(c)$. Therefore $e(v)Q_0(\epsilon(c))=0$ for all $v\in V$ distinct from $s(c)$: $Q_0(\epsilon(c))$ is proportional to $e(s(c))$. The constant of proportionality is non zero because $Q$ is supposed to be faithful. The equality $e(v)=\sum_{s(c)=v}\epsilon(c)$ gives the equality $\sum_{s(c)=v}p(c)=1$.
\end{proof}

\begin{thm}\label{fd expectation} Let $Q:\underline M\to M$ be a faithful conditional expectation on a finite dimensional C*-algebra and let $A$ be Cartan subalgebra of $M$. We let $(X,R)$ be the associated equivalence relation. Then,
\begin{enumerate}
\item there exists $\underline A$ Cartan subalgebra of $\underline M$ such that $(M,A)\subset (\underline M,\underline A)$ is a Cartan pairs inclusion compatible with $Q$.
\item any isomorphism $\Phi:  M\to  C^*(R)$ carrying $A$ onto $C_0(X)$ can be extended to an an isomorphism $\underline\Phi: \underline M\to  C^*(\underline R)$ carrying $Q$ into the model expectation $Q_p: C^*(\underline R)\to C^*(R)$ constructed from  the graph $(V,E,\underline V)$ of the inclusion, the spectrum $X$ of  $A$ and the transition probability $p$ of \lemref{transition probability}.
\end{enumerate} 
\end{thm}

\begin{proof} The first assertion is \lemref{key lemma}. We fix a Cartan subalgebra $\underline A$ satisfying (i). We recall that the spectrum $\underline X$ of $\underline A$ can be identified with the fibered product $X\times_V E$ and that the spectrum of $M'\cap\underline A$ can be identified with $E$. We also recall from \lemref{commutant} that $(M',M'\cap\underline A)$ is a Cartan pair defining the equivalence relation $R'$ on $E$. We pick  a matrix unit $(\epsilon(a,b))_{(a,b)\in R'}$ for the Cartan pair $(M',M'\cap\underline A)$.  Let $\Phi:  M\to  C^*(R)$ be an isomorphism carrying $A$ onto $C_0(X)$. There exists a unique matrix unit $(e(x,y))_{(x,y)\in R}$ for the Cartan pair $(M,A)$ which is sent by $\Phi$ onto the canonical matrix unit of $C^*(R)$. We define
$\underline e(xa,yb)=e(x,y)\epsilon(a,b)$ if $r(x)=r(y)=s(a)=s(b)$ and $r(a)=r(b)$. This defines a partial matrix unit on a subequivalence relation of $\underline R$. According to \lemref{partial}, it can be completed into a full matrix unit $(\underline e(xa,yb))_{(xa, yb)\in\underline R}$. The isomorphism $\underline \Phi: {\underline M}\to C^*({\underline R})$ defined by this matrix unit extends $\Phi$ and satisfies $\underline \Phi\circ Q= Q_p\circ\underline \Phi$.
\end{proof}

In particular, the theorem shows that our path model of a conditional expectation gives every faithful conditional expectation. We can recover from this theorem the main result of \cite{dav:expectation}, namely every conditional expectation on a finite dimensional C*-algebra can be written as a pinching followed by slicing and averaging: one introduces an intermediate level $V_1$  in the inclusion graph $(V,E,\underline V)$, whose vertices label the edges (thus $V_1=E$). The graph $(V,E,\underline V)$ is then written as the concatenation of two graphs $(V,E_1,V_1)$ and $(V_1,E_2,\underline V)$. In the first graph, the vertices of $V_1$ receive a single edge. In the second graph, the vertices of $V_1$ emit a single edge. With the ingredient $X_1=\underline X$ and $(V_1,E_2,\underline V)$, our recipe gives the inclusion $C^*(R_1)\subset C^*(\underline R)$ where $(xa,yb)\in R_1$ if and only if $a=b$. The transition probability $p_2\equiv 1$  gives the restriction map $Q_2: C^*(\underline R)\to C^*(R_1)$ as its associated conditional expectation. It is a pinching: in other words, it is of the form 
$$Q_2(\underline f)=\sum_{c\in E}\epsilon(c)\underline f \epsilon(c).$$
The conditional expectation $Q_1:C^*(R_1)\to C^*(R)$ is an averaging: for every $v\in V$
$$Q_1(f)(x,y)=\sum_{s(c)=v}p(c)f(xc,yc)\quad {\rm for}\quad r(x)=r(y)=v.$$
In \cite{ac:krieger}, A. Connes uses a similar decomposition of an inclusion of type I von Neumann algebras to construct inclusions of Cartan pairs.

\section{Random walks on discrete Bratteli diagrams.}

We first recall the classical definition of a Bratteli diagram. 
\begin{defn}\label{Bratteli}
A {\it Bratteli diagram} is a directed graph $(V,E)$ where the set of vertices $V=\coprod_{n=0}^\infty V(n)$ and the set of edges $E=\coprod_{n=1}^\infty E(n)$ are graded. For each  $n\ge 1$, $s(E(n))=V(n-1)$ and $r(E(n))=V(n)$, where $s(e)$ and $r(e)$ are respectively the source and the range of the edge $e$.
\end{defn}

We assume that each level of vertices $V(n)$ is at most countable; we also assumes that each vertex emits finitely many but at least one edge and that each vertex of a level $n\ge 1$ receives finitely many but at least one edge.

\begin{defn}\label{transition} Let $(V,E)$ be a Bratteli diagram.
\begin{itemize} 
\item A {\it transition probability} is a map $p$ assigning to each vertex $v\in V$ a probability measure $p(v)$ on the set of edges $E_v=s^{-1}(v)$ emanating from $v$. We shall view $p$ as a map $p:E\to\R$ such that for all $v\in V$, $\sum_{s(e)=v}p(e)=1$. We shall denote by $p_n$ its restriction to $E(n)$.
\item An {\it initial probability measure} is a probability measure $\nu_0$ on the set of initial vertices $V(0)$.
\item A {\it random walk} is a pair $(p,\nu_0)$, where $p$ is a transition probability and $\nu_0$ is an initial probability measure.
\end{itemize} 
\end{defn}

We shall assume in this section that $p$ and $\nu_0$ have full support, in the sense that $p(e)>0$ for all $e\in E$ and $\mu_0(v)>0$ for all $v\in V(0)$.\\ A Bratteli diagram $(V,E)$ defines an \'etale equivalence relation $(X,R)$ called the {\it tail equivalence relation} of the diagram: $X$ is the set of infinite paths $x=e_1e_2\ldots$ where $e_n\in E(n)$ and $r(e_n)=s(e_{n+1})$. It is a locally compact Hausdorff totally disconnected space admitting the cylinders
$$Z(a)=\{ae_{n+1}e_{n+2}\ldots\}$$
where $a=a_1a_2\ldots a_n$ is a finite path (we assume implicitly that $a_i\in E(i)$), as a base of compact open subsets.  Two infinite paths $x=e_1e_2\ldots$ and $y=f_1f_2\ldots$ are tail equivalent if there exists $n$ such that $e_i=f_i$ for $i>n$. Its graph $R$ is a locally compact Hausdorff totally disconnected space admitting the cylinders
$$Z(a,b)=\{(ae_{n+1}e_{n+2}\ldots,be_{n+1}e_{n+2}\ldots)\}$$
where $(a,b)$ is a pair of equivalent finite paths: this means that they have same length $n$ and same range $r(a)=r(b)$, where we define $r(a_1a_2\ldots a_n)=r(a_n)\in V(n)$.\\
A random walk on a Bratteli diagram defines a measure on the path space $X$; it is a particular case of the well-known construction of Markov measures.

\begin{prop} Given a random walk $(p,\nu_0)$ on a Bratteli diagram $(V,E)$, there is a unique probability measure $\mu$ on $X$ , called the Markov measure of the random walk whose values on cylinder sets is given by
$\,\mu(Z(a))=\nu_0(s(a))p(a)\,$  where, for the finite path $a=a_1a_2\ldots a_n$, 
 $p(a)=p_1(a_1)p_2(a_2)\ldots p_n(a_n)\quad {\rm and}\quad s(a)=s(a_1)$.
\end{prop}

As observed in \cite[Section 3.2]{ren:AF}, Markov measures are quasi-invariant under the tail equivalence relation. Let us recall that a measure $\mu$ on $X$ is quasi-invariant under the equivalence relation $R$ if the measures $r^*\mu$ and $s^*\mu$ on $R$ are equivalent, where $\int f d (r^*\mu)=\int\sum_y f(x,y)d\mu(x)$ for $f\in C_c(R)$ and $s^*\mu$ is similarly defined. Then its Radon-Nikodym derivative $D_\mu=d(r^*\mu)/d(s^*\mu)$ is a cocycle, i.e. it satisfies $D_\mu(x,y)D_\mu(y,z)=D_\mu(x,z)$ for a.e. $(x,y,z)\in R^{(2)}$. Here is a way to construct cocycles on the tail equivalence relation $R$ of a Bratteli diagram $(V,E)$.

\begin{defn}\label{quasi-product cocycle} Let $G$ be a group. A map $D:R\to G$ is called a {\it quasi-product cocycle} if there exists a map $q: E\to G$, called a potential, such that for all pairs of equivalent finite paths $(a,b)$ and all $(az,bz)\in Z(a,b)$, $D(az,bz)=q(a)q(b)^{-1}$ and where, as before, $q(a_1a_2\ldots a_n)=q(a_1)q(a_2)\ldots q(a_n)$.
\end{defn}

Since a quasi-product cocycle is locally constant, it takes at most countably many values and it is continuous. The following result, which is a simple observation, is essential here.

\begin{prop}\label{D-measure} \cite[Proposition 3.3]{ren:AF} Let $(p,\nu_0)$ be a random walk  on a Bratteli diagram $(V,E)$.
\begin{enumerate}
\item The associated Markov measure $\mu$ is quasi-invariant under the tail equivalence relation $R$
\item Its Radon-Nikodym derivative $D_\mu$ is the quasi-product cocycle given by the potential $q=(q_n)$ defined  by the relation
$$\nu_{n-1}(s(e))p_n(e)=q_n(e)\nu_n(r(e))\quad{\rm for}\quad e\in E(n).$$
where $\nu_n$ is the distribution of the random walk on $V(n)$, defined inductively by $\nu_n(w)=\sum_{r(e)=w}p_n(e)\nu_{n-1}(s(e))$ for $w\in V(n)$.
\end{enumerate}
\end{prop}

Note that the Radon-Nikodym derivative depends only on the potential $q$. This potential $q$ has a simple probabilitstic interpretation: it is the {\it cotransition probability} of the random walk: let $e$ be an edge in $E(n)$ with range $r(e)=w$, then $q(e)$ is the probability that the random walk passes through $e$ given that it is at $w$ at time $n$.  In his recent papers \cite{ver:equipped,ver:graded graphs}, A.~Vershik also emphasizes the importance of cotransition probabilities in the asymptotic study of random walks on Bratteli diagrams. I thank him for the reference  \cite{dyn:exit} on the subject. Cotransition probabilities are called ``{\it backward transition probabilities}''  in \cite{ks:markov}. Note that the cotransition probability $q$ depends not only on the transition probability $p$ but also on the initial measure $\nu_0$.  As shown by the next example, different random walks may share the same cotransition probability.

\begin{ex}\label{Pascal}
Random walks on the Pascal triangle. It is the time development of the simple random walk on $\Z$. Here, the Bratteli diagram is $(V,E)$ where
$$V(n)=\{(n,k): k=0,1,\ldots, n\};\quad E(n)=\{(n-1,k,\epsilon): k=0,1,\ldots, n-1; \epsilon\in\{0,1\}\}.$$
We consider the random walk defined by the transition probability 
$$p_n(n-1,k)=(1-t)\delta_{(n-1,k,0)} + t \delta_{(n-1,k,1)}$$
where $0<t<1$. Since $V(0)$ has a single vertex, the initial measure $\nu_0$ is the point mass at this vertex. The infinite path can be written as the infinite product $X=\prod_{n=1}^\infty\{0,1\}$. Then, the Markov measure is the product measure $\mu_t=\prod_{n=1}^\infty((1-t)\delta_0+t\delta_1)$. An elementary computation gives the cotransition probability
$$q_n(n,k)=(1-\frac{k}{n})\,\delta_{(n-1,k,0)} + \frac{k}{n}\, \delta_{(n-1,k-1,1)}$$
It does not depend on $t$. For a finite path $\epsilon_1\epsilon_2\ldots\epsilon_n$ ending at $(n,k=\epsilon_1+\ldots+\epsilon_n)$, one has
$$q(\epsilon_1\epsilon_2\ldots\epsilon_n)={n\choose k}^{-1}.$$
One deduces that the Radon-Nikodym of $\mu_t$ is $D\equiv 1$. In other words, the measures $\mu_t$ are invariant under the tail equivalence relation on $(V,E)$. It is a well-known result. It is also well-known (see for example \cite[Example 4.2]{ren:AP}) that these are the extremal invariant probability measures (one has to add $\mu_0$ and $\mu_1$ which we have excluded from our discussion).

\end{ex}

We use now the construction given by Feldman and Moore in \cite{fm:relations}: since $(X,R,\mu)$ is a countable standard measured equivalence relation, one can construct its von Neumann algebra ${\mathcal M}=W^*(X,R,\mu)$ and its state $\varphi=\mu\circ P$, where $P$ is the expectation of ${\mathcal M}$ onto ${\mathcal A}=L^\infty(X,\mu)$, which is normal and faithful. By construction, ${\mathcal M}$ acts on the Hilbert space $L^2(R,s^*\mu)$. This representation is standard. It is known that the modular operator $\Delta$ of $\varphi$ is the operator of multiplication by $D_\mu$ and that the modular automorphism $\sigma_t^\varphi$ is implemented by the operator of multiplication by $D_\mu^{it}$.  A. Connes has given the following characterization of the pairs $({\mathcal M},\varphi)$ arising from this construction.

\begin{thm}\label{Markov states} \cite[Theorem 1]{ac:krieger}  Let $\varphi$ be a faithful normal state on a von Neumann algebra ${\mathcal M}$. Then \tfae
\begin{enumerate}
\item there exists an increasing sequence $(M_n)$ of finite dimensional subalgebras  stable under the automorphism group $\sigma$ of $\varphi$ whose union is weakly dense in ${\mathcal M}$;
\item there exists a Bratteli diagram $(V,E)$ and a random walk $(p,\nu_0)$ on it such that the pair $({\mathcal M},\varphi)$ is isomorphic to $(W^*(X,R,\mu), \mu\circ P)$, where $R$ is the tail equivalence relation on the infinite path space $X$ of the diagram and $\mu$ is the Markov measure of the random walk.
\end{enumerate}
\end{thm}

The original theorem of Connes is in terms of Krieger's factors. It is an intermediate step to show that all hyperfinite type III$_0$ factors are Krieger's factors. We consider here von Neumann algebras arising from hyperfinite measured equivalence relations rather than Krieger's factors. This makes the statement easier to prove.
\vskip5mm
\begin{proof}{\it (ii)$\Rightarrow$(i)}. We assume that ${\mathcal M}=W^*(X,R,\mu)$ as above.
We let $M_n$ be the subalgebra of ${\mathcal M}$ generated by the characteristic functions ${\bf 1}_{Z(a,b)}$, where $(a,b)$ is a pair of joining paths of length $n$. Since the Radon-Nikodym derivative $D_\mu$ is constant on the cylinder sets $Z(a,b)$, $A_n$ is stable under the automorphism group of $\varphi_\mu$. Since $Z(a,b)$ is the disjoint union of $Z(ae,be)$'s where $e\in E(n+1)$ and $s(e)=r(a)=r(b)$, we have the inclusion $M_n\subset M_{n+1}$. The elements of the union $M_\infty$ of the $M_n$'s are the locally constant functions with compact support. Since $M_\infty$ is dense in $C_c(R)$ with respect to the inductive limit topology, it is dense in the weak topology. Since $C_c(R)$ is weakly dense in ${\mathcal M}$, so is $M_\infty$.
\vskip5mm
{\it (i)$\Rightarrow$(ii)}. Let $(M_n)_{n\in\N}$ be as in (i). Without loss of generality, we may assume that $M_0=\C 1$. Since for all $n$,  $M_n$ is stable under $\sigma$, the modular automorphism of the restriction $\varphi_n$ to $M_n$ is the restriction $\sigma_n$ of $\sigma$ to $M_n$. Since for all $n\ge 1$, $M_{n-1}$ is invariant under $\sigma_n$,   there exists a faithful expectation $Q_{n-1,n}: M_n\to M_{n-1}$ such that $\varphi_n=\varphi_{n-1}\circ Q_{n-1,n}$. We use inductively  \thmref{fd expectation}, to construct an increasing sequence $(A_n)$ of abelian subalgebras such that for all $n\ge 1$,  $(M_{n-1},A_{n-1})\subset (M_n,A_n)$ is a Cartan pair inclusion compatible with the conditional expectation $Q_{n-1,n}$. The construction is initialized by the only possible choice $A_0=M_0$. Thus we obtain for each $n\in\N$ the spectrum $X_n$ of $A_n$ and for each $n\ge 1$ the graph $(V(n-1), E(n), V(n))$ of the inclusion $M_{n-1}\subset M_n$ and the transition probability $p_n:E(n)\to\R_+^*$. From the same theorem, we obtain for all $n$ an isomorphism $\Phi_n: M_n\to C^*(R_n)$ sending $A_n$ to $C(X_n)$, where $(X_n,R_n)$ is the equivalence relation defined by $(M_n,A_n)$, such that $\Phi_n$ extends $\Phi_{n-1}$ and carrying the conditional expectation $Q_{n-1,n}:M_n\to M_{n-1}$ into the model conditional expectation $F_{p_n}: C^*(R_n)\to C^*(R_{n-1})$. Again, the construction is initialized by the only possible isomorphism $\Phi_0:M_0\to C(X_0)$. Since the conditional expectations $P_n: M_n\to A_n$ satisfy $Q_{n-1,n}\circ P_n=P_{n-1}\circ Q_{n-1,n}$, we have $\varphi_n=\nu_n\circ P_n$, where $\nu_n$ is the restriction of $\varphi_n$ to $A_n$. The Bratteli diagram of the increasing sequence $(M_n)$ of finite dimensional algebras is $(V=\coprod_{n=0}^\infty V(n), E=\coprod_{n=1}^\infty E(n))$. The transition probabilities $p_n:E(n)\to\R_+^*$ define a transition probability on $E$. The initial measure $\nu_0$ is the point mass at the unique point of $X_0$. This defines the random walk. We let $X$ be its infinite path space, $R$ be the tail equivalence and $\mu$ be the Markov measure of the random walk. The isomorphisms $\Phi_n:M_n\to C^*(R_n)$ extend to an isomorphism $\Phi_\infty: M_\infty\to C_{00}(R)$, where $M_\infty$ is the union of the $M_n$'s and $C_{00}(R)$ is the $*$-algebra of locally constant functions with compact support. This isomorphism carries the restriction of the state $\varphi$ to the restriction of the state $\mu\circ P$, where $P$ is the expectation of  $W^*(R)$ onto $L^\infty(X,\mu)$. Since both von Neumann algebras ${\mathcal M}$ and $W^*(X,R,\mu)$ can be obtained from the GNS representation of these states, $\Phi_\infty$ extends to a normal $*$-isomorphism $\Phi: {\mathcal M}\to W^*(X,R,\mu)$ which sends the weak closure $\mathcal A$ of the union of the $A_n$'s to $L^\infty (X,\mu)$ and $\varphi$ to $\mu\circ P$.

\end{proof}

\begin{rem} The von Neumann algebra  ${\mathcal M}$ in the theorem is a factor if and only if the Markov measure $\mu$ is ergodic under the tail equivalence relation. Its flow of weights can be computed from the Radon-Nikodym derivative $D_\mu$, or more concretely, from the cotransition probability $q$. For example, the above transition probability $p_t$ on the Pascal triangle gives the hyperfinite II$_1$ factor. We shall return to the probabilistic identification of the flow of weights in the last section.
\end{rem}

\begin{rem} States on AF-algebras constructed from a random walk on a Bratteli diagram are called quasi-product states in \cite{eva:quasiproduct}. Do we have a characterization (besides condition (i) of the theorem) of the normal faithful states on a hyperfinite von Neumann algebra which can be described as quasi-product states? Necessarily, theses states are almost periodic (their modular operators are diagonalizable) and their centralizers contain a Cartan subalgebra. In part II of \cite{ac:krieger}, A. Connes shows that any faithful semifinite normal weight on a hyperfinite factor of type $III_0$ whose modular operator $\Delta$ is diagonalizable and such that $1$ is isolated in its spectrum and its point spectrum contained in $\Q$ satisfies condition (i) of the theorem (with $M_n$ type $I_\infty$ rather than finite-dimensional).
\end{rem}

\section{Markov chains and Bratteli diagrams}

In order to recover the general theory of time-dependent Markov chains (with discrete time), it is necessary to generalize the notion of Bratteli diagram. Indeed, the original definition is limited to Markov chains with at most countably many states. Generalized Bratteli diagrams have been considered before, mostly in the topological setting, and are part of the theory of topological graphs. Since we are considering objects of measure-theoretical nature, we choose the Borel setting. We assume implicitly that the Borel spaces are analytic.

\begin{defn} We say that a directed graph $(V,E)$ is a {\it Borel graph} if the sets of edges $E$ and the set of vertices $V$ are endowed with a Borel structure and the source and range maps are Borel. A {\it Borel Bratteli diagram} is a Bratteli diagram which is a Borel graph.
\end{defn}

Before extending \defnref{transition} to Borel Bratteli diagrams , we need to make precise our assumptions.

\begin{defn} Let $E,V$ be Borel spaces and let $s:E\to V$ be a Borel surjection. A Borel $s$-system of probability measures $p$ is a map assigning to each $v\in V$ a probability measure $p_v$ on $s^{-1}(v)$ and such that for all bounded Borel function $f$ on $E$, the map $v\to \int f dp_v$ is Borel.
\end{defn}

Given a Borel $s$-system $p$ and a probability measure $\nu$ on $V$, we can form the probability measure $\mu=\nu p$ on $E$, defined by $\int f d(\nu p)=\int (\int f dp_v)d\nu(v)$
for $f$ bounded Borel function on $E$. The measure $\nu$ is the image $s_*(\nu p)$ of $\nu p$ and $\mu=\nu p$ is the disintegration of $\mu$ along $s$. This construction requires in fact weaker assumptions on $p$:  it suffices to have $p_v$ defined for a.e. $v$ and the $\nu$-measurability of the function $v\to \int f dp_v$ for all bounded Borel function $f$ on $E$. However, the disintegration theorem of measures, as stated for example in \cite[Theorem 2.1]{hah:haar}, says that, conversely, given a probability measure $\mu$ on $E$, there exists a Borel $s$-system of probability measures $p$, where $\nu=s_*(\mu)$, such that $\mu=\nu p$, where $\nu=s_*(\mu)$. It is unique in the sense that if $\nu p=\nu p'$, then $p_v=p'_v$ for $\nu$-a.e.$v$. With an abuse of language, we shall say that $(\nu,p)$ is the disintegration of $\mu$ along $s$.

\begin{notation} Consider the $n$-th floor $V(n-1)\xleftarrow{s} E(n)\xrightarrow{r} V(n)$ of a Borel Bratteli diagram. A probability measure $\mu_n$ on $E(n)$ admits a disintegration $\mu_n=\nu_{n-1}p_n$ along $s$ and a disintegration $\mu_n=\nu_nq_n$ along $r$, where $\nu_{n-1}=s_*\mu_n$ and $\nu_n=r_*\mu_n$. This establishes a bijection between pairs $(\nu_{n-1},p_n)$, where $\nu_{n-1}$ is a probability measure on $V(n-1)$ and $p_n$ is a Borel system of probability measures along $s$ and pairs $(\nu_{n},q_n)$,where $\nu_n$ is a probability measure on $V(n)$ and $q_n$ is a Borel system of probability measures along $r$ , given by the equation $\nu_{n-1}p_n=\nu_nq_n$. 
\end{notation}

\begin{defn}\label{random walk} Let $(V,E)$  be a Borel Bratteli diagram.
\begin{itemize}
\item A {\it transition probability} $p$ is a Borel system of probability measures for the source map $s: E\to V$. It will be usually viewed as a sequence $p=(p_n)$ of  Borel systems of probability measures for $s: E(n)\to V(n-1)$.
\item A {\it cotransition probability} $q$ is a Borel system of probability measures for the range map $r: E\to V$. It will be usually viewed as a sequence $q=(q_n)$ of  Borel systems of probability measures for $r: E(n)\to V(n)$.
\item A {\it random walk} on $(V,E)$ is a sequence of probability measures $\mu_n$ on $E(n)$ which are compatible in the sense that for all $n\ge 1$, $r_*\mu_n=s_*\mu_{n+1}$. 
\item The measures $\nu_n=r_*\mu_n$ are called the {\it one-dimensional distributions} of the random walk. 
\item The measure $\nu_0=s_*\mu_1$ is called the {\it initial distribution} of the random walk. 
\end{itemize} 
\end{defn}

Let $(\mu_n)$ be a random walk on the Bratteli diagram $(V,E)$. The disintegration of $\mu_n$ along $s$ and $r$  gives respectively a pair $(\nu_{n-1},p_n)$ and a pair $(\nu_n,q_n)$ as above. Then $p=(p_n)$ [resp. $q=(q_n)$] is called the transition [cotransition] probability of the random walk.  The measures $\nu_n$ are called the one-dimensional distributions of the random walk. They satisfy the relations $\nu_{n-1}=s_*(\nu_nq_n)$ and $\nu_n=r_*(\nu_{n-1}p_n)$ for all $n\ge$. We say that the sequence $(\nu_n)$ is $q$-{\it compatible} and $p$-{\it compatible} respectively.

\begin{prop} Let $(V,E)$ be a Borel Bratteli diagram $(V,E)$.
\begin{enumerate}
\item Given a transition probability $p$ and a probability measure $\nu_0$ on $V(0)$, there exists a unique random walk on $(V,E)$ admitting $p$ as its transition probability and $\nu_0$ as its initial distribution.
\item Given a cotransition probability $q$ and  a sequence of probability measures $\nu_n$ on $V(n)$ such that $\nu_{n-1}=s_*(\nu_nq_n)$ for all $n\ge 1$, there exists a unique random walk on $(V,E)$ admitting $q$ as its cotransition probability and $(\nu_n)$ as its one-dimensional distribution.
\end{enumerate}
\end{prop}

\begin{proof} 
This is clear. In the first case, we define inductively $\mu_n=\nu_{n-1}p_n$. In the second case, we define $\mu_n=\nu_n q_n$.
\end{proof}

\begin{cor} Given a Borel Bratteli diagram $(V,E)$, there is a one-to-one correspondence between
\begin{enumerate}
\item pairs $(p,\nu_0)$, where $p$ is a transition probability and $\nu_0$ is a probability measure on $V(0)$;
\item pairs $(q, (\nu_n))$, where $q$ is a cotransition probability and $\nu_n$ is a $q$-compatible sequence of probability measures on $V(n)$
\end{enumerate}
given by the relation $\nu_{n-1}p_n=\nu_nq_n$.
\end{cor}

From now on, a random walk on $(V,E)$ will designate indifferently the measures $\mu_n$ on $E(n)$ as in \defnref{random walk}, the pair $(p,\nu_0)$ or the pair $(q, (\nu_n))$ as in the above corollary. We recall the construction of the Markov measure of a random walk (see \cite[V-1]{nev:proba}). As earlier, we introduce the infinite path space $X$. We let $X^{n}=E(1)*\ldots*E(n)$ denote the space of paths $e_1\ldots e_n$ of length $n$ endowed with the product Borel structure. Then $X=\projlim X^{n}$ is the projective limit with respect to the canonical projection $X^{n}\leftarrow X^{n+1}$. Given a random walk $(p,\nu_0)$, one first construct by induction a probability measure $\mu^{n}$ on $X^{n}$ such that
$$\int fd\mu^1=\int f(e_1)dp_v(e_1)d\nu_0(v),\quad\int f d\mu^{n}=\int f(e_1\ldots e_n)dp_{r(e_{n-1})}(e_n)d\mu^{n-1}(e_1\ldots e_{n-1})$$
The sequence of measures $(\mu^n)$ is consistent. Therefore, there exists a unique probability measure $\mu$ on $X$ whose image in $X^n$ is $\mu^n$. Note that the one-dimensional distribution $\nu_n$ on $V(n)$ is the image of $\mu^n$ by the range map $r:X^{n}\to V(n)$. It is also the image of $\mu$ by the map $r_n: X\to V(n)$ such that $r_n(e_1e_2\ldots)=r(e_n)$.

It remains to characterize the Markov measure $\mu$ on $X$ in terms of the cotransition probability $q$. We have seen that in the framework of the previous section, the Markov measure $\mu$ is quasi-invariant under the tail equivalence relation and its Radon-Nikodym derivative $D$ is the quasi-product cocycle defined by $q$. We then say that $\mu$ is a $D$-measure. The notion of quasi-product cocycle does not admit a straightforward generalization in the general framework. However, there exists (see \cite[Proposition 3.7]{ren:AP}) an equivalent definition of a $D$-measure (known in statistical mechanics as the Dobrushin-Lanford-Ruelle condition for Gibbs states) which can be easily extended. We let $X_{|n}$ be the space of infinite paths $e_{n+1}e_{n+2}\ldots$ starting at level $n$. The sequence  of quotient maps
$$X\xrightarrow{\pi_1} X_{|1}\xrightarrow{\pi_2} X_{|2}\xrightarrow{\pi_3}\ldots\xrightarrow{\pi_n} X_{|n}\xrightarrow{\pi_{|n+1}}\ldots$$
defines the tail equivalence relation $R$ on $X$: two infinite paths $x$ and $y$ are tail equivalent if and only if there exist $n$ such that $\pi_n\circ\ldots\pi_2\circ\pi_1(x)= \pi_n\circ\ldots\pi_2\circ\pi_1(y)$. A cotransition probability $q$ defines an inductive system of expectations
$$B(X)\xrightarrow{\tilde q_1} B(X_{|1})\xrightarrow{\tilde q_2}\ldots\xrightarrow{\tilde q_n} B(X_{|n})\xrightarrow{\tilde q_{n+1}}$$
where $B(Y)$ is the space of bounded complex-valued Borel functions on $Y$ and
$$\tilde q_n(f)(e_{n+1}e_{n+2}\ldots)=\int f(e_ne_{n+1}e_{n+2}\ldots)dq_n^{s(e_{n+1})}(e_n).$$

\begin{defn} Let $q$ be a cotransition probability on the Borel Bratteli diagram $(V,E)$.
A $q$-{\it measure} is a measure on the infinite path space $X$ which factors through all expectations $\tilde q_n\ldots \tilde q_2\tilde q_1$.
\end{defn}

Then we have the easy generalisation of \propref{D-measure}:

\begin{thm}\label{q-measure} Let $\mu$ be a probability measure on the infinite path space of a Borel Bratteli diagram $(V,E)$. Then \tfae
\begin{enumerate}
\item $\mu$ is a $q$-measure;
\item $\mu$ is the Markov measure of a random walk admitting $q$ as its cotransition probability.
\end{enumerate}
\end{thm}

\begin{proof} Let $\mu$ be a Markov measure with transition probability $p$, cotransition probability $q$ and one-dimensional distribution $(\nu_n)$. Let us show that $\mu$ factors through $\tilde q_m\ldots \tilde q_2\tilde q_1$ for all $m$. The measure $\mu_{|m}$ on $X_{|m}$, image of $\mu$ by $\pi_m\circ\ldots\pi_2\circ\pi_1$, is the Markov measure defined by the initial measure $\nu_m$ and the transition probability $(p_n), n>m$. Since measures on $X$ are uniquely determined by their values on cylinder sets, it suffices to show that for $n> m$,
$\mu^n=\mu_{|m}^n\tilde q_m\ldots \tilde q_2\tilde q_1$
where $\mu^n$ [resp. $\mu_{|m}^n$] is the measure of the random walk on $X^n=E(1)*\ldots *E(n)$ [resp. $X^n_{|m}=E(m+1)*\ldots *E(n)$]. But this is clear from the construction of the Markov measure.\\
Let $\mu$ be a $q$-measure. We define for all $n$ the measure $\nu_n$ as the image of $\mu$ by the map $r_n: X\to V(n)$ such that $r_n(e_1e_2\ldots )=r(e_n)$. Because of the relation $r_n=s_n\circ\pi_n\circ\ldots\pi_2\circ\pi_1$, it is also the image of $\mu_{|n}$ by the map $s_n: X_{|n}\to V(n)$ such that $s_n(e_{n+1}e_{n+2}\ldots)=s(e_{n+1})$. Since $\mu_{|n-1}=\mu_n\circ \tilde q_n$, $\nu_{|n-1}=s_*(\nu_nq_n)$: the sequence $(\nu_n)$ is $q$-compatible. Therefore it is the one-dimensional distribution of a Markov chain with cotransition probability $q$. The disintegration $\mu=\mu_{|n}(\tilde q_n\circ\ldots\circ\tilde q_1)$ gives the disintegration $(\pi^n)_*\mu=\nu_n(q_n\circ\ldots\circ q_1)$ where $\pi^n: X\to X^n$ is the projection $\pi^n(e_1e_e\ldots)=e_1\ldots e_n$. Therefore $(\pi^n)_*\mu$ agrees with the measure $\mu^n$ of the random walk. This suffices to conclude that $\mu$ is the Markov measure of the random walk.
\end{proof}

\begin{defn} Given a random walk with transition probability $p$ and one-dimensional distributions $(\nu_n)$ on a Bratteli diagram $(V,E)$, a {\it bounded harmonic sequence} is a sequence $(h_n)$ where $h_n$ belongs to $L^\infty(V(n),\nu_n)$, $h_{n-1}=p_n(h_n\circ r)$ and $\sup_n\|h_n\|_\infty <\infty$.
\end{defn}

\begin{notation} The bounded harmonic sequences, equipped with the norm $\sup_n\|h_n\|_\infty$, form a Banach space $H(p,\nu_0)$ which is the projective limit of the sequence
\begin{displaymath}
(M)\qquad L^\infty(V(0),\nu_0)\xleftarrow{p_1} L^\infty(V(1),\nu_1)\xleftarrow{p_2}\ldots\xleftarrow{p_n} L^\infty(V(n),\nu_n)\xleftarrow{p_{n+1}}\ldots
\end{displaymath}
where the maps are the expectations defined by the transition probability $p$. More precisely, with an abuse of notation, we define $p_n(h)=p_n(h\circ r)$ for $h\in L^\infty(V(n),\nu_n)$. We also note that $H(p,\nu_0)$ depends only on the measure class $[\nu_0]$ of $\nu_0$.
\end{notation}

We can now give the ergodic decompostion of a Markov measure $\mu$ under the tail equivalence $R$ on the infinite path space $X$ of a Bratteli diagram. Recall that $R$ is defined by the maps $\pi_n\circ\ldots\pi_2\circ\pi_1: X\to X_{|n}$. We say that a function $f$ on $X$ is invariant if for all $n$, there exists $f_n$ on $X_{|n}$ such that $f=f_n\circ\pi_n\circ\ldots\pi_2\circ\pi_1$. We denote by $L^\infty(X,\mu)^R$ the subalgebra of invariant elements of $L^\infty(X,\mu)$.

\begin{thm}\label{harmonic}\cite[Proposition V-2-2]{nev:proba} Let $\mu$ be the Markov measure of a random walk on a Borel Bratteli diagram $(V,E)$ having a transition probability $p$ and an initial measure $\nu_0$. Then the ordered Banach spaces $L^\infty(X,\mu)^R$ and $H(p,\nu_0)$ are naturally isomorphic.
\end{thm}

\begin{proof} Let $q$ be the cotransition probability of the random walk and $(\nu_n)$ be its one-dimensional distribution. A positive element $f$ of $L^\infty(X,\mu)^R$ defines a finite measure $\mu'=f\mu$ such that $\mu'\le M \mu$ where $M =\|f\|_\infty$. Since $f$ is invariant, $\mu'$ is a $q$-measure. Therefore the sequence of measures $\nu'_n=(r_n)_*\mu'$, where the map $r_n: X\to V(n)$ is the same as in the proof of \thmref{q-measure}, is $q$-compatible. Moreover, we have the inequality $\nu'_n\le M\nu_n$. Hence there exists $h_n\in L^\infty (V(n),\nu_n)$ such that $\nu'_n=h_n\nu_n$. We also have $\|h_n\|_\infty\le M$. Since $\nu_{n-1}p_n=\nu_nq_n$, the condition $\nu'_{n-1}=s_*(\nu'_nq_n)$ gives $h_{n-1}=p_n(h_n\circ r)$. Thus $h=(h_n)$ is a positive bounded harmonic sequence of norm at least $\|f\|_\infty$. Conversely, let $h=(h_n)$ be a positive bounded harmonic sequence. We define $\nu'_n=h_n\nu_n$. From the relations $h_{n-1}=p_n(h_n\circ r)$ and $\nu_{n-1}=s_*(\nu_nq_n)$, we deduce that  $\nu'_{n-1}=s_*(\nu'_nq_n)$ and that $\nu'_n(1)=\nu'_{n-1}(1)$. We set $M=\sup_n\|h_n\|_\infty$. Since the sequence $(\nu'_n)$ is $q$-compatible, there exists a random walk (we no longer have probability measures but finite measures of the same mass) and a Markov measure $\mu'$ admitting $q$ as cotransition probability and $(\nu'_n)$ as one-dimensional distribution. The condition $\nu'_n\le M\nu_n$ gives $\mu'^n\le M\mu^n$ where $\mu^n$ and $\mu'^n$ are the corresponding measures on $X^n$. Therefore $\mu'\le M\mu$. There exists a unique positive element $f\in L^\infty(X,\mu)$ such that $\mu'=f\mu$ and it satisfies $\|f\|_\infty\le M$. The first part shows that $\|f\|_\infty=M$. Thus this correspondence gives an isomorphism between the positive cones of $L^\infty(X,\mu)^R$ and $H(p,\nu_0)$ which extends to an isomorphism of the ordered Banach spaces.

\end{proof}

\begin{rem} The proof given here relies only on the Radon-Nikodym theorem and on the disintegration theorem of probability measures. The classical proof of \cite[Proposition V-2-2]{nev:proba} uses the martingale convergence theorem. It gives explicit formulas relating an element $f$ of $L^\infty(X,\mu)^R$ and a bounded harmonic sequence $(h_n)$ in $H(p,\nu_0)$:
\begin{eqnarray}
f(e_1e_2\ldots)&=&\lim_n h_n(r(e_n))\quad{\rm a.e.}\\
h_n&=& P_{|n}(f_n)
\end{eqnarray}
where $f_n\in B(X_{|n})$ comes from the factorization $f=f_n\circ\pi_n\circ\ldots\pi_2\circ\pi_1$ and $\mu_{|n}=\nu_nP_{|n}$ is the disintegration of $\mu_{|n}$ along the source map  $s_n: X_{|n}\to V(n)$.
\end{rem}

\begin{defn} The point realization of the commutative von Neumann algebra $L^\infty(X,\mu)^R$ is called the {\it tail boundary} of the random walk. It is a standard Borel space $P$ equipped with a probability measure $m$. By definition, $L^\infty(X,\mu)^R=L^\infty(P,m)$ and $m$ is the restriction of the measure $\mu$.
\end{defn}

\begin{rem} This terminology is not a standard one. In \cite{cw:random}, it is called the Poisson boundary of the time dependent random walk. There is a dual construction of the tail boundary, given in \cite{gh:matrix-valued}, based on the inductive sequence of Banach spaces
\begin{displaymath}
(L)\qquad L^1(V(0),\nu_0)\xrightarrow{q_1} L^1(V(1),\nu_1)\xrightarrow{q_2}\ldots\xrightarrow{q_n} L^1(V(n),\nu_n)\xrightarrow{q_{n+1}}\ldots
\end{displaymath}
where $q_n(f)=q_n(f\circ s)$ for $f\in L^1(V(n-1),\nu_{n-1})$. Its inductive limit can be written $L^1(P,m)$ because it is an $L$-space. The adjoint of $q_n: L^1(V(n-1),\nu_{n-1})\to L^1(V(n),\nu_n)$ is the map $p_n: L^\infty(V(n),\nu_n)\to L^\infty(V(n-1),\nu_{n-1})$ defined earlier. Therefore, the projective limit $H(p,\nu_0)$ of the sequence $(M)$ is the dual of $L^1(P,m)$.
\end{rem}

\section{Matrix-valued random walks on groups}

\begin{defn} A {\it matrix-valued random walk} on a Borel group $G$ is given by the following data:
\begin{enumerate}
\item a random walk $(p,\nu_0)$ on a Borel Bratteli diagram $(V,E)$,
\item a Borel map $\rho: E\to G$.
\end{enumerate}
\end{defn}

Here is the construction of the Poisson boundary of a matrix-valued random walk. One first construct a new Bratteli diagram, called the skew-product of $(V, E, \rho)$. This is a particular case of a construction given for graphs or higher rank graphs in \cite{kp:graphs}.

\begin{defn} Let $\Gamma=(V,E)$ be a Borel Bratteli diagram and let $\rho$ be a Borel map from $E$ to a Borel group $G$. The  {\it skew-product}  $\Gamma(\rho)$ is the Bratteli diagram $(V\times G, E\times G)$ where $s(e,g)=(s(e),g)$ and $r(e,g)=(r(e),g\rho(e))$.
\end{defn}

The skew-product $\Gamma(\rho)$ carries a compatible Borel structure and an automorphic action of $G$, given by $h(v,g)=(v,hg)$ and $h(e,g)=(e,hg)$. The infinite path space of $\Gamma(\rho)$ can be identified with $X\times G$, where $X$ is the infinite path space of $\Gamma$: we associate to $(e_1e_2\ldots , g)$ the path $(e_1,g)(e_2,g\rho(e_1)),\ldots$. The map $\rho:E\to G$ defines a $G$-valued quasi-product cocycle  $c$ on the tail equivalence relation $R$ on $X$ of $(V,E)$ according to \defnref{quasi-product cocycle}. The tail equivalence relation of the skew-product Bratteli diagram $\Gamma(\rho)$ is the skew-product equivalence relation $R(c)$ on $X\times G$, as defined in \cite[Definition I.1.6]{ren:approach}. Explicitly
$$(e_1e_2\ldots , g)\sim (f_1f_2\ldots , h)\Leftrightarrow \exists N: {\rm for}\, n\ge N, e_n=f_n\quad g\rho(e_1)\ldots\rho(e_n)=h\rho(f_1)\ldots \rho(f_n)$$
Let $\mu$ be the Markov measure defined by the random walk $(p,\nu_0)$ on $(V,E)$. Then the Markov measure defined by the random walk $(\tilde p,\nu_0\times \lambda)$, where  $\tilde p_{(v,g)}=p_v\times \delta_g$ and $\lambda$ is a finite measure equivalent to the Haar measure of $G$, is  $\mu\times\lambda$. 
\begin{defn} Let $(V, E,\rho:E\to G,p,\nu_0)$ be a matrix-valued random walk on a locally compact group $G$. Its {\it Poisson boundary} is the point realization $(P,m)$ of 
$$L^\infty(X\times G,\mu\times\lambda)^{R(c)}\simeq H(\tilde p,\nu_0\times\lambda).$$
It is a measured $G$-space.
\end{defn}

\begin{ex} Time-dependent random walks on a group. This is the case when $(V,E)$ is a UHF diagram, i.e. there is only one vertex at each level. Let us assume  that $G$ is a discrete group and that $(p_n)$ is a sequence of probability measures on $G$ with finite support $G_n$. The UHF diagram $(V,E)$ is defined by $E(n)=G_n$. The map $\rho_n: E(n)\to G$ is the inclusion map. The infinite path space of $(V,E)$ is the product space $X=\prod G_n$ and its Markov measure is the product measure $\mu=\prod p_n$.  The one-dimensional distributions of the skew-product Bratteli diagram are the measures $\tilde\nu_n=\lambda*p_1*p_2\ldots*p_n$ on $G$. For a concrete example, let $G=\Z$ with probability measures $p_n=(1-t)\delta_0+t\delta_1$ for all $n$. Then $(V,E)$ is the UHF($2^\infty$) diagram. Choosing $\delta_0$ as initial measure rather than a finite measure equivalent to the counting measure on $\Z$, one gets the random walk of \exref{Pascal}. This is the kernel diagram rather than the skew product diagram described above. It can be checked directly that the bounded harmonic sequences are constant; in other words, the Poisson boundary of the time dependent random walk is trivial.
\end{ex}

\begin{ex} Flow of weights of hyperfinite von Neumann algebras.  Let $\varphi$ be a faithful normal state on a von Neumann algebra ${\mathcal M}$ satisfying the equivalent conditions of \thmref{Markov states}. Thus, there exists a Bratteli diagram $(V,E)$ and a random walk $(p,\nu_0)$ on it such that the pair $({\mathcal M},\varphi)$ is isomorphic to $(W^*(X,R,\mu), \mu\circ P)$, where $R$ is the tail equivalence relation on the infinite path space $X$ of the diagram and $\mu$ is the Markov measure of the random walk. Since the Radon-Nikodym $D_\mu$ is the quasi-product cocycle defined by the cotransition probability $q:E\to\R_+^*$ of the random walk, the flow of weights of $M$ is the Poisson boundary of the matrix-valued random walk on $R_+^*$ defined by $q$. 
\end{ex}

\vskip 5mm
{\it Acknowledgements.}  I thank T.~Giordano and A.~Vershik for fruitful discussions.


\vskip3mm




\begin{thebibliography}{10}

\bibitem{aeg:continuous} S.~Adams, G.~Elliott and T.~Giordano: {\it Amenable actions of groups}, Trans. Amer. Math. Soc., {\bf 344} (1994), no. 2, 803--322.

\bibitem{bra:AF} O.~Bratteli: {\it Inductive limits of finite dimensional C*-algebras}, Trans. Amer. Math. Soc., {\bf 171} (1971), 195--234.

\bibitem{ac:krieger} A.-Connes: {\it On hyperfinite factors of type III$_0$ and Krieger's factors}, J.~Funct. Anal., 
{\bf 18} (1975), 318--327.

\bibitem{cw:random} A.~Connes and E. J.~Woods: {\it Hyperfinite von Neumann algebras and Poisson boundaries of time dependent random walks}, Pacific J. Math. {\bf 137} (1989), no 2, 225--243.

\bibitem{dav:expectation} C.~Davis: {\it Various averaging operations on subalgebras}, Illinois J. Math., {\bf 3} (1959), 538--553.

\bibitem{dyn:exit} E.~Dynkin: {\it The exit space of a Markov process}, Russian Math. Surveys {\bf 24}, No. 4, (1969), 89--152.

\bibitem{eg:discrete}   G.~Elliott and T.~Giordano: {\it Amenable actions of discrete groups}, Ergodic Theory Dyn. System, {\bf 13} (1993), no.2, 289--318.

\bibitem{eva:quasiproduct} D.~E.~Evans: {\it Quasi-product states on
C$^*$-algebras}, Lecture Notes in Mathematics, No. 1132,
Springer-Verlag, Berlin-New York, 1985, 129-151. 

\bibitem{fm:relations} J.~Feldman and C.~Moore: {\it Ergodic equivalence relations,
cohomologies, von Neumann algebras, I and II},
Trans. Amer. Math. Soc., {\bf 234} (1977), 289-359.

\bibitem{gh:matrix-valued}  T.~Giordano and D.~Handelman: {\it Matrix-valued random walks and variations on property AT},
M\"unster J. of Math. 1(2008),15--72.

\bibitem{hah:haar}  P.~Hahn: {\it Haar measures for measure groupoids}, Trans. Amer. Math. Soc., {\bf 242}, (1978) ,
1--33.

\bibitem{ks:markov}  J.G.~Kemeny and J.L.~Snell : {\it Finite Markov chains}, Van Nostrand, Princeton N.J., 1960.

\bibitem{kp:graphs}  A.~Kumjian and D.~Pask: {\it  Higher rank graph C*-algebras}, New York J. Math. 6 (2000), 1?20. 

\bibitem{nev:proba}  J.~Neveu: {\it Bases math\'ematiques du calcul des probabilit\'es}, Masson \& Cie, Paris, 1964.

\bibitem{ren:approach}  J.~Renault: {\it A groupoid approach to
$C^*$-algebras}, Lecture Notes in Mathematics, Vol.~{\bf 793}
Springer-Verlag Berlin, Heidelberg, New York (1980).

\bibitem{ren:AF}  J.~Renault: {\it  AF equivalence relations and their
cocycles}, Operator Algebras and Mathematical Physics, Conference Proceedings, Constanza 2001, the Theta Foundation (2003), 365-377.

\bibitem{ren:AP}  J.~Renault: {\it  The Radon-Nikodym problem for approximately proper equivalence relations}, Ergodic Theory Dyn. System, {\bf 25} (2005), 1643--1672.

\bibitem{ver:equipped}  A.~Vershik {\it Equipped graded graphs, projective limits of simplices, and their boundaries}, Zapiski Nauchn. Semin. POMI  {\bf 432}, (2015), 83--104,
English translation to appear in J. Math. Sci.

\bibitem{ver:graded graphs}  A.~Vershik {\it Asymptotic theory of path spaces of graded graphs and its applications}, Japanese J. Math. {\bf 11}, No. 2, (2016), 151--218 .
 

\end{thebibliography}
\end{document}